\documentclass[leqno,12pt]{article}
\usepackage{amsfonts,amsmath,amssymb}
\usepackage{amsthm}
\usepackage{enumerate}
\usepackage{graphics}
\usepackage{float}%
\usepackage{enumerate}%
\usepackage[all]{xy}

\numberwithin{equation}{section} %

\theoremstyle{plain} %
  \newtheorem{theorem}{Theorem}[section] %
  \newtheorem{proposition}[theorem]{Proposition}%
  \newtheorem{lemma}[theorem]{Lemma}%
  \newtheorem{corollary}[theorem]{Corollary}%

\theoremstyle{definition} %

\newtheorem{def-prop}[theorem]{Definition-Proposition}
\theoremstyle{remark} %
  \newtheorem{remark}[theorem]{Remark}%
\newcommand{\dcoh}[2]{H_{\overline \partial}^{#1}({#2})}
\newcommand{\invHom}[3]{\operatorname{Hom}_{#1}({#2},{#3})}
\newcommand{\Kfdrep}[2]{F^{#1}(#2)}

\newcommand\Sol{{\mathcal {S}}\!{\text{\it{ol}}}\,}

\begin{document}
\title
{Holomorphic Laplacian on the Lie ball and the Penrose transform}
\author{Hideko SEKIGUCHI}
\date{} %
\maketitle %

\begin{abstract}
We prove 
 that any holomorphic function $f$
 on the Lie ball of even dimension 
satisfying $\Delta f=0$
 is obtained uniquely
 by the higher-dimensional Penrose transform
 of a Dolbeault cohomology
 for a twisted line bundle 
 of a certain domain 
 of the Grassmannian of isotropic subspaces.  
To overcome the difficulties
 arising from that the line bundle parameter 
 is outside the {\it{good range}}, 
 we use some techniques from algebraic representation theory.  
\end{abstract}


\setcounter
{section}{0}

\section{Introduction}

We consider the holomorphic Laplacian
 $\Delta=\frac{\partial^2}{\partial z_1^2} + \cdots +\frac{\partial^2}{\partial z_n^2}$ on the Lie ball
\[
D:= \{z \in {\mathbb{C}}^n: 
 |z \, {}^{t\!}z|^2+1 - 2 \overline z\, {}^{t\!}z>0, 
 |z \, {}^{t\!}z|<1
\}, 
\]
which is the bounded symmetric domain of type D IV
 in the {\'E}.\ Cartan classification.

The goal of this article is to prove
 that any holomorphic function $f$ on $D$
 satisfying $\Delta f=0$
 can be obtained uniquely
 as the higher-dimensional Penrose transform 
 of a Dolbeault cohomology 
 of a non-compact complex manifold $X=SO_0(2,2m)/U(1,m)$
 when $n=2m$.

To formulate our main results, 
 let $m>1$, $G:=SO_0(2,2m)$ 
 be the identity component of the indefinite orthogonal group of 
 signature $(2, 2m)$, 
 $K=SO(2) \times SO(2m)$ a maximal compact subgroup, 
 and $\theta$ the corresponding Cartan involution of $G$.  
Let $T$ be a maximal torus of $K$, 
 ${\mathfrak{t}}$ its Lie algebra, 
 and ${\mathfrak{t}}^{\vee}$ the dual space.  
We take the standard basis 
 $\{e_0, e_1, \dots, e_m\}$
 in $\sqrt{-1}{\mathfrak{t}}^{\vee}$
 such that 
 $\Delta({\mathfrak{g}}_{\mathbb{C}}, {\mathfrak{t}}_{\mathbb{C}})
 =\{\pm e_i \pm e_j:0 \le i <j \le m\}$. 
One has a decomposition of the complexified Lie algebra 
$
   {\mathfrak{g}}_{\mathbb{C}}={\mathfrak{so}}(2m+2, {\mathbb{C}})
  ={\mathfrak{p}}_-+{\mathfrak{k}}_{\mathbb{C}}+{\mathfrak{p}}_+
$
 as a ${\mathfrak{k}}_{\mathbb{C}}$-module
 with $\Delta({\mathfrak{p}}_+, {\mathfrak{t}}_{\mathbb{C}}):=\{e_0 \pm e_j: 1 \le j \le m\}$.  
The bounded symmetric domain $D$ may be identified 
 with the Harish-Chandra realization of $G/K$
 in ${\mathbb{C}}^{2m} \simeq {\mathfrak{p}}_- \subset G_{\mathbb{C}}/K \exp ({\mathfrak{p}}_+)$.

We set ${\bf{1}}_{m+1} = e_0 + \dots +e_{m}
 \in \sqrt{-1}{\mathfrak{t}}^{\vee}$, 
 and define a $\theta$-stable parabolic subalgebra
 ${\mathfrak{q}}= {\mathfrak{l}}_{\mathbb{C}} + {\mathfrak{u}}$
 with ${\mathfrak{l}}_{\mathbb{C}} \supset {\mathfrak{t}}_{\mathbb{C}}$
 such that the roots $\alpha$
 for ${\mathfrak{l}}_{\mathbb{C}}$ 
 and ${\mathfrak{u}}$ are given 
 by $\langle \alpha, {\bf{1}}_{m+1} \rangle=0$
 and $\langle \alpha, {\bf{1}}_{m+1} \rangle>0$, 
 respectively.  
Let $Q$ be the parabolic subgroup 
 of the complexified Lie group $G_{\mathbb{C}}=SO(2m+2, {\mathbb{C}})$
 with Lie algebra ${\mathfrak{q}}$, 
 and $L:=G \cap Q \simeq U(1,m)$.  
Then the homogeneous space $X:=G/L=SO_0(2,2m)/U(1,m)$ is identified
 with the set of indefinite Hermitian structures
 on ${\mathbb{R}}^{2m+2}$ 
 of signature $(1,m)$, 
 and becomes a complex manifold
 as an open set of $G_{\mathbb{C}}/Q \simeq SO(2m+2)/U(m+1)$, 
 the Grassmannian of isotropic subspaces
 of ${\mathbb{C}}^{2m+2}$
 equipped with non-degenerate quadratic form.  
For $\lambda \in {\mathbb{Z}}$, 
 let ${\mathbb{C}}_{\lambda}$ denote the holomorphic character 
 of $L_{\mathbb{C}} \simeq GL(m+1, {\mathbb{C}})$
 given by ${\det}^{\lambda}$, 
 and we form a $G_{\mathbb{C}}$-equivariant holomorphic line bundle
${\mathcal{L}}_{\lambda}:=G_{\mathbb{C}} \times_Q {\mathbb{C}}_{\lambda}$
 over $G_{\mathbb{C}}/Q$.  
We shall use the same letter ${\mathcal{L}}_{\lambda}$
 to denote its restriction  ${\mathcal{L}}_{\lambda}|_X \simeq G \times_L {\mathbb{C}}_{\lambda}$
 to the open subset $X \simeq G/L$ of $G_{\mathbb{C}}/Q$.  
With this notation, 
 the canonical bundle $\Omega_X$
 of $X$
 is given by ${\mathcal{L}}_m$.

Let $\dcoh j{X, {\mathcal{L}}_{\lambda}}$ be
 the $j$-th Dolbeault cohomology group
 with coefficients in ${\mathcal{L}}_{\lambda}$, 
 which carries a natural Fr{\'e}chet topology
 by the closed range theorem 
 of the $\bar \partial$ operator \cite{xwong}.  
We set 
\[
  \Sol(D, \Delta)=
  \{f \in {\mathcal{O}}(D): \Delta f =0\}, 
\]
 and equip it with the topology of uniform convergence
 on every compact sets.  
We prove:

\begin{theorem}
[see Theorem {\ref{thm:main}}]
\label{thm:23021140}
Let ${\mathcal{R}}$ be the cohomological integral transform 
 (Penrose transform)
 defined in \eqref{eqn:defPen} below.  
Then ${\mathcal{R}}$ gives a topological $G$-isomorphism: 
\[
  {\mathcal{R}} \colon \dcoh{m(m-1)}{X, {\mathcal{L}}_{m-1}} \to 
 \Sol(D, \Delta).  
\]
\end{theorem}

In the case $m=2$, 
 via the double covering $SU(2,2) \to G= SO_0(2,4)$, 
 the group $G$ is of type A, 
 $X$ is biholomorphic to $SU(2,2)/U(1,2)$, 
 and $D$ is biholomorphic to the 4-dimensional bounded symmetric domain 
 of type A III.  
In this case, 
 the bijectivity of ${\mathcal{R}}$ in Theorem \ref{thm:23021140}
 was first proved
 in \cite{xEPW}, 
 and later generalized in \cite{xhseki96}
 by a different approach.

Theorem \ref{thm:23021140} in the case $m \ge 3$
 consists of the following assertions:
\newline
(a)\enspace
the range of ${\mathcal{R}}$ satisfies the differential equation;
\newline
(b)\,(surjectivity)\enspace
the Penrose transform ${\mathcal{R}}$ constructs all the solutions;
\newline
(c)\,(injectivity)\enspace
the kernel of ${\mathcal{R}}$ is zero;
\newline
(d)\,(non-vanishing)\enspace
the $m(m-1)$-th cohomology does not vanish;
\newline
(e)\,(cohomological purity)\enspace
the $j$-th cohomology vanishes if $j \ne m(m-1)$;
\newline
(f)\,(topology)\enspace
${\mathcal{R}}$ is not only a bijection 
 but also a topological isomorphism.

There have been various approaches to (a) in certain settings, 
{\it{e.g.,}} Eastwood--Penrose--Wells \cite{xEPW}, 
Mantini \cite{xmahw},
 Marastoni--Tanisaki \cite{xmrtani03}
 for sufficiently positive parameter $\lambda$.  
We note that the representations on the cohomologies 
 with coefficients in the \lq\lq{good range}\rq\rq\
 in the sense of Vogan \cite{xvu}
 are well-understood
 by the Beilinson--Bernstein theory \cite{xbebe}
 or by the algebraic representation theory 
({\it{e.g.,}} \cite{xvu, xvr}).  
However, 
 Theorem \ref{thm:23021140} needs to treat the parameter $\lambda$
 {\bf{outside the good range}}
 for which the general theory does not apply.  
It should be noted 
 that there are {\bf{counterexamples}}
 to what we may expect from (c)--(d):
\vskip0.5pc
\par\noindent
(c)$'$\enspace
the Penrose transform may have an infinite-dimensional kernel;
\newline
(d)$'$\enspace
the Dolbeault cohomologies of all degrees may vanish.  
\vskip 0.5pc
\par\noindent
See, {\it{e.g.,}} \cite{xhseki02, xhseki13} for (c)$'$;
 Kobayashi \cite{xkupq}
 and Trapa \cite{T01} for (d)$'$
 for some classical groups.  
 In fact, 
 (d) is a special case of a long-standing problem
 in algebraic representation theory
 about when Zuckerman's derived functor module 
 $A_{\mathfrak{q}}(\lambda) \ne 0$
 for singular $\lambda$
 outside the good range.

For $m \ge 3$, 
 there are two minimal representations
 of the group $G=SO_0(2,2m)$
 in the sense that their annihilator is the Joseph ideal
 \cite{xJoseph}
 of the enveloping algebra $U({\mathfrak{g}}_{\mathbb{C}})$, 
 and they are dual to each other in our setting, 
 see  {\it{e.g.,}} \cite{xtamori}.  
Theorem \ref{thm:23021140} gives their complex-geometric realization.  

\begin{corollary}
\label{cor:230217}
For $m \ge 3$, 
 the two minimal representations of $G=SO_0(2,2m)$
 are realized in the cohomologies $\dcoh {m(m-1)} {X, {\mathcal{L}}_{m-1}}$
 where ${\mathfrak{q}}$ is taken to be 
$
   {\mathfrak{l}}_{\mathbb{C}}+{\mathfrak{u}}
$
 or 
$
   {\mathfrak{l}}_{\mathbb{C}}+{\mathfrak{u}}^-
$.  
\end{corollary}

\begin{remark}
Kobayashi--{\O}rsted \cite{KOr98}
 proposed yet another complex-geometric realization
 of the minimal representation of $O(p,q)$
 by using Dolbeault cohomologies 
 on a non-compact complex manifold
 when $p+q$ is even
 and $p,q \ge 2$.  
Our complex manifold $X$ is different from the one in \cite{KOr98}
 with $p=2$.

Last but not least, 
 the representation of $G$ on the Fr{\'e}chet space 
 of Dolbeault cohomologies
 in Theorem \ref{thm:23021140} contains
 the unique Hilbert space
 as its dense subspace
 on which $G$ acts as an irreducible unitary representation, 
 to be denoted by $\pi$, 
 by \cite{xvu} and by Proposition \ref{prop:irr}.  
In his paper \cite{vanDijk}, 
 van Dijk classified \lq\lq{generalized Gelfand pairs}\rq\rq\ $(G,H)$
 under the assumption
 that $G/H$ is a semisimple symmetric space
 of rank one.  
In particular, 
 $(G,H)=(SO_0(2,2n), SO_0(2,2n-1))$
 is a generalized Gelfand pair, 
 and thus $\dim \operatorname{Hom}_G(\pi, {\mathcal{D}}'(G/H)) \le 1$.  
\end{remark}

\section{Penrose transform}
\label{sec:Pen}
The morphism in Theorem \ref{thm:23021140}
 is the higher-dimensional Penrose transform, 
 of which we review quickly from \cite{xhseki02}
 the definition 
 adapted to our specific situation.

Let $K$ be a maximal compact subgroup of a linear reductive Lie group $G$, 
 $\theta$ a complexified Cartan involution, 
 ${\mathfrak{q}}={\mathfrak{l}}_{\mathbb{C}}+{\mathfrak{u}}$
 a $\theta$-stable parabolic subalgebra of ${\mathfrak{g}}_{\mathbb{C}}$, 
 and $X=G/L$ the open $G$-orbit in the flag variety 
 $G_{\mathbb{C}}/Q$
 through the origin $o=e Q$, 
 see {\it{e.g.,}} \cite{xkono}.  
We consider a compact submanifold
 $C:=K/L \cap K \simeq K_{\mathbb{C}}/Q \cap K_{\mathbb{C}}$ in $G/L$, 
 and write $\iota \colon C \hookrightarrow X$
 for the natural embedding.  
Let $S$ denote the complex dimension of $C$, 
 and $T$ be a maximal torus of $L \cap K$, 
 hence that of $K$, too.  
We take a positive system $\Delta^+({\mathfrak{k}}_{\mathbb{C}}, {\mathfrak{t}}_{\mathbb{C}})$
 containing the weights $\Delta({\mathfrak{u}} \cap {\mathfrak{k}}_{\mathbb{C}}, {\mathfrak{t}}_{\mathbb{C}})$.  
For a dominant character $\mu$ of $T$, 
 we let $V_{\mu}$ denote the irreducible $K$-module
 with highest weight $\mu$, 
 and form a $G$-equivariant vector bundle
 ${\mathcal{V}}_{\mu}:=G \times_K V_{\mu}$
 over the Riemannian symmetric space $G/K$.  
We write $\ell_g$ for the action of $g \in G$
 on the line bundle ${\mathcal{L}}_{\lambda}$ over $G/L$.  
Then the natural map
\begin{equation*}
{\mathcal{E}}^{0,S}(G/L, {\mathcal{L}}_{\lambda}) \times G \to {\mathcal{E}}^{0,S}(K/L \cap K, \iota^{\ast} {\mathcal{L}}_{\lambda}), 
\quad
 (\alpha, g) \mapsto \iota^{\ast}\ell_g^{\ast} \alpha
\end{equation*}
induces the one for Dolbeault cohomologies:
\[
   \dcoh S {G/L, {\mathcal{L}}_{\lambda}} \times G
   \to
   \dcoh S {K/L \cap K, \iota^{\ast} {\mathcal{L}}_{\lambda}},
\quad
  ([\alpha], g) \mapsto [\iota^{\ast} \ell_g^{\ast} \alpha].  
\]
By the Borel--Weil--Bott theorem, 
 the target space is finite-dimensional, 
 and is $K$-isomorphic
 to the irreducible representation $V_{\mu_{\lambda}}$
 as far as $
\mu_{\lambda}:={\mathbb{C}}_{\lambda} \otimes \Lambda^S({\mathfrak{k}}_{\mathbb{C}}/({\mathfrak{q}} \cap {\mathfrak{k}}_{\mathbb{C}}))
$  is $\Delta^+({\mathfrak{k}}_{\mathbb{C}}, {\mathfrak{t}}_{\mathbb{C}})$-dominant.  
In turn, 
 the above map yields a continuous $G$-homomorphism:
\begin{equation*}
  {\mathcal{R}} \colon 
  H_{\overline \partial}^S(G/L, {\mathcal{L}}_{\lambda}) \to C^{\infty}(G/K, {\mathcal{V}}_{\mu_{\lambda}}), 
\quad
  [\alpha] \mapsto (g \mapsto [\iota^{\ast} \ell_g^{\ast} \alpha]), 
\end{equation*}
 which is referred 
 to as a (higher-dimensional) Penrose transform, 
 or to the Penrose transform in short
 (\cite[Thm.\ 2.6]{xhseki02}).

In our setting,
 the compact submanifold $C \simeq SO(2m)/U(m)$, 
 $S=m(m-1)$, 
 and ${\mathfrak{u}}=({\mathfrak{u}} \cap {\mathfrak{k}}_{\mathbb{C}}) \oplus ({\mathfrak{u}} \cap {\mathfrak{p}}_{\mathbb{C}})$
 with 
\begin{equation}
\label{eqn:uk}
\Delta({\mathfrak{u}}\cap{\mathfrak{k}}_{\mathbb{C}})
=\{e_i + e_j: 1 \le i < j \le m\}, 
\quad
\Delta({\mathfrak{u}}\cap{\mathfrak{p}}_{\mathbb{C}})
=\{e_0 + e_j: 1 \le j \le m\}.  
\end{equation}

Then halves the sums of the roots
 in $\Delta({\mathfrak{u}}, {\mathfrak{t}}_{\mathbb{C}})$
 and $\Delta({\mathfrak{u}} \cap {\mathfrak{k}}_{\mathbb{C}}, {\mathfrak{t}}_{\mathbb{C}})$ are given respectively by 
\[
\rho({\mathfrak{u}})
=\frac m 2 {\bf{1}}_{m+1}, 
\quad
 \rho({\mathfrak{u}} \cap
 {\mathfrak{k}}_{\mathbb{C}})=0 \oplus \frac{m-1}2{\bf{1}}_m
\]
in the standard coordinates
 ${\mathfrak{t}}_{\mathbb{C}}^{\ast} \simeq {\mathbb{C}}^{m+1}$.  
Since $\mu_{\lambda}=(m-1, 0, \dots,0)$
 for $\lambda=(m-1) {\bf{1}}_{m+1}$, 
  one has a $G$-intertwining operator
\begin{equation}
\label{eqn:defPen}
  {\mathcal{R}} \colon \dcoh{m(m-1)}{G/L, {\mathcal{L}}_{m-1}} \to 
C^{\infty}(G/K, {\mathcal{V}}_{m-1}).  
\end{equation}

Here, by an abuse of notation 
 we write ${\mathcal{V}}_{m-1}$
 for the line bundle ${\mathcal{V}}_{(m-1, 0,\dots,0)}$
 over $G/K$, 
 which is isomorphic to $(\Omega_{G/K})^{\frac{m-1}{2m}}$
 where $\Omega_{G/K}$ is the canonical bundle of $G/K$.  
Trivializing the line bundle ${\mathcal{V}}_{m-1}$
 via the Harish-Chandra realization 
 $G/K \overset \sim \longrightarrow D \subset {\mathfrak{p}}_-$
 of the Hermitian symmetric space, 
 one may identify ${\mathcal{F}}(G/K, {\mathcal{V}}_{m-1})$
 with ${\mathcal{F}}(D)$
 for ${\mathcal{F}}=C^{\infty}$ or ${\mathcal{O}}$.

Now our theorem is formulated as follows.  
\begin{theorem}
\label{thm:main}
The higher-dimensional Penrose transform ${\mathcal{R}}$ in \eqref{eqn:defPen} is injective, 
 and its image coincides with $\Sol(D, \Delta)$ 
 via the identification 
 ${\mathcal{O}}(D) \simeq {\mathcal{O}}(G/K, {\mathcal{V}}_{m-1})$
 where $G=SO_0(2,2m)$.  
Moreover, 
 ${\mathcal{R}}$ gives a $G$-equivariant topological isomorphism:
\[
  {\mathcal{R}} \colon 
  \dcoh {m(m-1)}{G/L, {\mathcal{L}}_{m-1}} \overset \sim \longrightarrow \Sol(D, \Delta).  
\]
\end{theorem}

\begin{remark}
\label{rem:23}
For $G=SO_0(2,2m+1)$ or its covering group, 
 the higher-dimensional Penrose transform 
 can be defined in a similar geometric setting, 
 however, 
 we do not expect an analogous theorem holds.  
For instance, 
 if $m=1$, 
 then via the double covering 
 $Sp(2,{\mathbb{R}}) \to SO_0(2,3)$, 
 one sees
 that the range of the Penrose transform does not satisfy
 the differential equation
 $\Delta f=0$
 as was proved in \cite{xhseki02}
 in the $Sp(n,{\mathbb{R}})$ case.  
 \end{remark}

\section{Differential operators on $G/K$}

In this section 
 we analyze the space $\Sol(D, \Delta)$ 
 and see that it is realized naturally as a $G$-submodule
 of ${\mathcal{O}}(G/K, {\mathcal{V}}_{m-1})$
 and compute its $K$-type formula.

We begin with some useful results about parabolic Verma modules.  
Let ${\mathfrak{g}}_{\mathbb{C}}={\mathfrak{so}}(2m+2, {\mathbb{C}})$
 with $m>1$, 
 ${\mathbb{C}}_{\lambda}$ be a character 
 of ${\mathfrak{k}}_{\mathbb{C}}+{\mathfrak{p}}_+$
 that takes the form $(\lambda,0,\dots, 0)$
 on ${\mathfrak{t}}_{\mathbb{C}} \simeq {\mathbb{C}}^{m+1}$
 for $\lambda \in {\mathbb{C}}$.  
By \cite[Lem.\ 10.1]{KS15}, 
 for $\lambda, \nu \in {\mathbb{C}}$, 
 one has
\[
\invHom{{\mathfrak{g}}_{\mathbb{C}}}
{U({\mathfrak{g}}_{\mathbb{C}}) \otimes_{U({\mathfrak{k}}_{\mathbb{C}}+{\mathfrak{p}}_+)} {\mathbb{C}}_{-\nu}}{U({\mathfrak{g}}_{\mathbb{C}}) \otimes_{U({\mathfrak{k}}_{\mathbb{C}}+{\mathfrak{p}}_+)} {\mathbb{C}}_{-\lambda}
}
 \ne \{0\}
\]
if and only if
\begin{equation}
\label{eqn:lmdnu}
  (\lambda,\nu)=(m - \ell, m + \ell)\qquad
\text{for some $\ell \in {\mathbb{N}}$.}
\end{equation}

In turn, 
 it follows from the duality theorem, 
 see {\it{e.g.,}}
\cite[Thm.\ 2.12]{KP1}
 applied to $G_{\mathbb{C}}=G_{\mathbb{C}}'=SO(2m+2, {\mathbb{C}})$, 
 that there is a holomorphic ${\mathfrak{g}}_{\mathbb{C}}$-intertwining 
 differential operator
 between the ${\mathfrak{g}}_{\mathbb{C}}$-equivariant sheaves 
 ${\mathcal{O}}(G_{\mathbb{C}}/K_{\mathbb{C}} \exp ({\mathfrak{p}}_+), {\mathcal{V}}_{m-\ell})$
 and ${\mathcal{O}}(G_{\mathbb{C}}/K_{\mathbb{C}} \exp ({\mathfrak{p}}_+), {\mathcal{V}}_{m+\ell})$.  
Such an operator is unique up to scalar multiplication, 
 and is given 
 as the $\ell$-th power of the holomorphic Laplacian 
 on the Bruhat open cell
 ${\mathfrak{p}}_-$.  
\begin{equation}
\label{eqn:Lapl}
  \Delta^{\ell} \colon {\mathcal{O}}(G_{\mathbb{C}}/K_{\mathbb{C}} \exp ({\mathfrak{p}}_+), {\mathcal{V}}_{m-\ell}) 
\to 
{\mathcal{O}}(G_{\mathbb{C}}/K_{\mathbb{C}} \exp ({\mathfrak{p}}_+), 
{\mathcal{V}}_{m + \ell}), 
\end{equation}
 see \cite{KP2}.  
See also Remark \ref{rem:33} for analogous operators
 in another real form.

In the case $\ell =1$, 
 one has a $G$-intertwining operator
\begin{equation}
\label{eqn:cpxLap}
   \Delta \colon {\mathcal{O}}(D, {\mathcal{V}}_{m-1})
          \to {\mathcal{O}}(D, {\mathcal{V}}_{m+1}).  
\end{equation}

Here is the $K$-structure of the kernel:
\begin{proposition}
\label{prop:3.1}
The kernel of the holomorphic Laplacian $\Delta$ is a $G$-submodule
 of ${\mathcal{O}}(D, {\mathcal{V}}_{m-1})$
 with the following $K$-structure:
\[
  \Sol(D,\Delta)_{K\operatorname{-finite}}
  \simeq
  \bigoplus_{\ell=0}^{\infty}
  {\mathbb{C}}_{\ell+m-1} \boxtimes \Kfdrep{SO(2m)}{\ell, 0, \dots, 0}.  
\]
\end{proposition}

\begin{remark}
As we shall see in Theorem \ref{thm:main} and Proposition \ref{prop:irr}, 
 the $G$-module $\Sol(D, \Delta)$ is irreducible.  
\end{remark}

\begin{proof}[Proof of Proposition \ref{prop:3.1}]
Let $\operatorname{Pol}^{\ell}({\mathfrak{p}}_{-})$ denote 
 the space of homogeneous polynomials in ${\mathfrak{p}}_-$ of degree $\ell$.  
We set 
\[
  {\mathcal{H}}({\mathfrak{p}}_{-})
  :=
  \{f \in \operatorname{Pol}({\mathfrak{p}}_{-})
  :
  \Delta f =0\}, 
\quad
  {\mathcal{H}}^{\ell}({\mathfrak{p}}_{-})
  := 
  {\mathcal{H}}({\mathfrak{p}}_{-}) \cap \operatorname{Pol}^{\ell}({\mathfrak{p}}_{-}).   
\]
Then $SO(2m)$ acts irreducibly 
 on ${\mathcal{H}}^{\ell}({\mathfrak{p}}_{-})$
 for every $\ell \in {\mathbb{N}}$
 when $m>1$, 
 and its highest weight is given by $(\ell, 0, \dots, 0)$.  
Since the first factor $SO(2)$ of $K$ acts 
 on $\operatorname{Pol}^{\ell}({\mathfrak{p}}_{-}) \simeq S^{\ell}({\mathfrak{p}}_+)$
 as the character ${\mathbb{C}}_{\ell}$, 
 the irreducible decomposition of the $K$-module
${\mathcal{H}}({\mathfrak{p}}_{-})$ is given as $\bigoplus_{\ell=0}^{\infty}
{\mathcal{H}}^{\ell}({\mathfrak{p}}_{-}) \simeq \bigoplus_{\ell=0}^{\infty} {\mathbb{C}}_{\ell} \boxtimes \Kfdrep {SO(2m)}{\ell, 0, \dots, 0}$.  
Now the proposition follows from the observation
 that the $K$-module structure 
 on the underlying $({\mathfrak{g}}_{\mathbb{C}}, K)$-module 
 ${\mathcal{O}}(G/K, {\mathcal{V}}_{m-1})_{K\operatorname{-finite}}$
 is given as the multiplicity-free direct sum $\operatorname{Pol}({\mathfrak{p}}_{-}) \otimes ({\mathbb{C}}
_{m-1} \boxtimes {\bf{1}}) \simeq 
 \overset{\infty}{\underset{\ell=0}\bigoplus}S^{\ell}({\mathfrak{p}}_+)\otimes ({\mathbb{C}}
_{m-1} \boxtimes {\bf{1}})$.  
\end{proof}

\begin{remark}
\label{rem:33}
Let $P_{\mathbb{R}}$ a minimal parabolic subgroup
 of $G_{\mathbb{R}}=SO_0(2m+1,1)$.  
Then $G=SO_0(2,2m)$ and $G_{\mathbb{R}}=SO_0(2m+1, 1)$
 have the common complexifications
 $G_{\mathbb{C}}=SO(2m+2, {\mathbb{C}})$.  
We may regard $(G_{\mathbb{R}}, P_{\mathbb{R}})$
 as a real from of $(G_{\mathbb{C}}, Q)$.  
Let $I(\lambda)=\operatorname{Ind}_{P_{\mathbb{R}}}^{G_{\mathbb{R}}}({\mathbb{C}}_{\lambda})$
 be the unnormalized spherical principal series representation 
 induced from a character ${\mathbb{C}}_{\lambda}$
 for $\lambda \in {\mathbb{C}}$.  
Our parametrization is taken to be the same
 with the one in the monograph \cite{KS15}
 so that the trivial representation ${\bf{1}}$ of $G_{\mathbb{R}}$
 occurs as the unique subrepresentation $I(0)$
 and also as the unique quotient of $I(n)$, 
 see \cite[(2.11)]{KS15}.  
Then the Knapp--Stein intertwining operator
$
I(\lambda) \to I(n-\lambda)
$
 has a pole at $\lambda \in\{\frac n 2, \frac n 2-1, \frac n 2-2, \dots\}$
 and its residue is a scalar multiple
 of the $\ell$-th power of the (Riemannian) Laplacian $\Delta$
 on the open Bruhat cell ${\mathbb{R}}^n$ 
 \cite[(4.29) and Remark 10.3]{KS15}
 if we set $\ell:=\frac n 2-\lambda \in {\mathbb{N}}$.  
In particular, 
 for $n=2m$, 
 the residue operator
\[
  \Delta \colon I(m-1) \to I(m+1)
\]
 is a $G_{\mathbb{R}}$-intertwining operator. 
This operator may be regarded
 as a \lq\lq{real form}\rq\rq\ of the holomorphic differential operator
 \eqref{eqn:cpxLap}.  
\end{remark}

\begin{remark}
\label{rem:23021609}
With the notation of the classification \cite{EHW83}
 of irreducible unitarizable lowest weight modules 
 for ${\mathfrak{g}}={\mathfrak{s o}}(2,n)$, 
 one has $A(\lambda_0)=\frac n 2$
 and $B(\lambda_0)=n-1$
 if $\lambda_0=(1-n)e_0$.   
Accordingly, 
 the lowest weight $({\mathfrak{g}}_{\mathbb{C}}, K)$-module $L(-(\lambda_0+ze_0))=L((n-1-z)e_0)$
 is unitarizable
 if and only if $z \le 0$
 or $z \in \{\frac n 2, n-1\}$.  
Its ${\mathfrak{Z}}({\mathfrak{g}}_{\mathbb{C}})$-infinitesimal character is given by $(n-1-z)e_0 - \rho_G$. 
The underlying $({\mathfrak{g}}_{\mathbb{C}}, K)$-module of $\Sol(D, \Delta)$
 is isomorphic
 to $L((\frac n 2-1)e_0)$ corresponding
 to the first reduction point $z=A(\lambda_0)$.   
\end{remark}

\section{Generalized Blattner formula}
\label{sec:Blattner}

In this section we examine the $K$-type formula
 of the Dolbeault cohomology group.  

We recall $K=SO(2) \times SO(2m)$.  
Irreducible $K$-modules are parametrized by 
 $\mu_0 \in {\mathbb{Z}}$
 and $\mu =(\mu_1, \dots, \mu_m) \in {\mathbb{Z}}^m$
satisfying 
 $\mu_1 \ge \cdots \ge \mu_{m-1} \ge |\mu_m|$.  
We write ${\mathbb{C}}_{\mu_0} \boxtimes F^{SO(2m)}(\mu)$
 for the irreducible $K$-module
 with highest weight $(\mu_0; \mu_1, \dots, \mu_m)$.

\begin{proposition}
\label{prop:23021120}
As a $K$-module, 
 $\dcoh{m(m-1)}{G/L, {\mathcal{L}}_{-1}}_{K\operatorname{-finite}}$ is 
multiplicity-free, 
 and its $K$-type formula is given by 
\[
   \bigoplus_{\ell=0}^{\infty} {\mathbb{C}}_{\ell+m-1} \boxtimes \Kfdrep{SO(2m)}{\ell, 0, \dots,0}.  
\]
\end{proposition}

\begin{remark}
In connection to the theory of visible actions 
 on complex manifolds (\cite{xrims40}), 
 Kobayashi raised a cohomological multiplicity-free conjecture
 in the general setting 
for branching problems of Zuckerman's derived functor modules
 \cite[Conj.\ 4.2]{xkzuckerman11}.  
We observe
 that both $(G,K)$ 
 and $(G,L)$ are reductive symmetric pairs, 
 hence Proposition \ref{prop:23021120} gives
 an evidence
 of his conjecture.  
\end{remark}

The proof of Proposition \ref{prop:23021120} is based
 on a combinatorial computation
 of a generalized Blattner formula
 as in \cite[Chap.\ 4]{xkupq}. 
Since the statement of this section is 
 purely algebraic, 
 we need only the underlying $({\mathfrak{g}}_{\mathbb{C}}, K)$-modules
 of the cohomologies, 
 namely, 
 Zuckerman's derived functor modules.  
As an algebraic analogue of the Dolbeault cohomology
 with coefficients in a $G$-equivariant holomorphic vector bundle
 over the complex manifold $G/L$, 
Zuckerman introduced a derived functor 
$
   {\mathcal{R}}_{\mathfrak {q}}^j \equiv 
  ({\mathcal{R}}_{\mathfrak {q}}^{{\mathfrak{g}}_{\mathbb{C}}})^j
$
 ($j \in {\mathbb{N}}$) as a cohomological parabolic induction.  
We follow \cite{xkupq, xkronse,xvr} for the normalization
 so that 
 ${\mathcal {R}}_{{\mathfrak {q}}}^j$ is a covariant functor
 from the category of metaplectic $({\mathfrak {l}}_{\mathbb{C}}, 
(L \cap K)\!\widetilde{\hphantom{m}})$-modules
 to the category of $({\mathfrak {g}}_{\mathbb{C}}, K)$-modules
 and that 
 if $\nu \in {\mathfrak {h}}_{\mathbb{C}}^*/W({\mathfrak {l}}_{\mathbb{C}})$
 is the ${\mathfrak{Z}}({\mathfrak {l}}_{\mathbb{C}})$-infinitesimal character
 of an $({\mathfrak{l}}_{\mathbb{C}}, (L \cap K)\!\widetilde{\hphantom{m}})$-module $V$
 then the ${\mathfrak{Z}}({\mathfrak {g}}_{\mathbb{C}})$-infinitesimal character of ${\mathcal{R}}_{{\mathfrak {q}}}^{S}(V)$
 equals 
$
   \nu \in {\mathfrak {h}}_{\mathbb{C}}^*/W({\mathfrak {g}}_{\mathbb{C}})
$.

Retain the setting as in Section \ref{sec:Pen}.  
By an abuse of notation, 
 we write ${\mathbb{C}}_{\lambda-\rho({\mathfrak{u}})}$
 for the metaplectic $({\mathfrak{l}}_{\mathbb{C}}, (L \cap K)\!\widetilde{\hphantom{m}})$-character
 with its differential $\lambda {\bf{1}}_{m+1}-\rho({\mathfrak{u}})$
 when $\lambda \in{\mathbb{Z}}$.  
Since $\lambda{\bf{1}}_{m+1}-\rho({\mathfrak{u}})
=(\lambda-\frac m 2){{\bf{1}}_{m+1}}$, 
 one has the following $({\mathfrak{g}}_{\mathbb{C}}, K)$-isomorphisms
 \cite{xwong}:
\begin{equation}
\label{eqn:DR}
  \dcoh j{G/L, {\mathcal{L}}_{\lambda}}_{K\operatorname{-finite}}
  \simeq
   {\mathcal{R}}_{\mathfrak{q}}^{j} ({\mathbb{C}}_{\lambda-\rho({\mathfrak{u}})})
  =
  {\mathcal{R}}_{\mathfrak{q}}^{j} ({\mathbb{C}}_{(\lambda-\frac m 2){\bf{1}}_{m+1}})
 \qquad\text{for all $j$}, 
\end{equation}
which has ${\mathfrak{Z}}({\mathfrak{g}}_{\mathbb{C}})$-infinitesimal character
$
   (\lambda-\frac m 2){\bf{1}}_{m+1} + \rho_{\mathfrak{l}} 
  = \lambda {\bf{1}}_{m+1} + (0,  -1, \dots, -m).  
$
Then the generalized Blattner formula
 for Zuckerman's derived functor modules asserts
 the following identity
\begin{multline}
\label{eqn:Bla}
  \sum_i (-1)^i \dim \invHom K {\pi}{{\mathcal{R}}_{{\mathfrak{q}}}^{S-i}({\mathbb{C}}_{\lambda-\rho({\mathfrak{u}})})}
\\
=
 \sum_j (-1)^j \dim \invHom {L \cap K} {H^j({\mathfrak{u}} \cap {\mathfrak{k}}_{\mathbb{C}}, \pi)}{S({\mathfrak{u}} \cap {\mathfrak{p}}_{\mathbb{C}}) \otimes {\mathbb{C}}_{\mu_{\lambda}}}, 
\end{multline}
 for any $\pi \in \widehat K$,
 where $S=\dim_{\mathbb{C}}({\mathfrak{u}}\cap{\mathfrak{k}}_{\mathbb{C}})=m(m-1)$
 and $S({\mathfrak{u}} \cap {\mathfrak{p}}_{\mathbb{C}})$ denotes 
 the space of symmetric tensors.

We are particularly interested in the case $\lambda= m-1$
 with $m \ge 2$.  
Then  $\mu_{\lambda}={\mathbb{C}}_{\lambda} \otimes \Lambda^S({\mathfrak{k}}_{\mathbb{C}}/({\mathfrak{q}} \cap {\mathfrak{k}}_{\mathbb{C}}))$ 
 amounts to $(m-1) \oplus 0{\bf{1}}_m$.  
On the other hand, 
 the parameter of the metaplectic $({\mathfrak{l}}_{\mathbb{C}}, (L \cap K)\widetilde{\hphantom{m}})$-character
 ${\mathbb{C}}_{\lambda-\rho({\mathfrak{u}})}
 ={\mathbb{C}}_{(\frac m 2-1){\bf{1}}_{m+1}}$, 
 see \eqref{eqn:DR}, 
 lies 
 in the weakly fair range
 with respect to ${\mathfrak{q}}$, 
 namely, 
 $\langle \lambda{\bf{1}}_{m+1}-\rho({\mathfrak{u}}), \alpha \rangle \ge 0$
 for any $\alpha \in \Delta({\mathfrak{u}})$.  
The general theory \cite{xwong}
 guarantees
 neither the irreducibility 
 nor the non-vanishing of ${\mathcal{R}}_{{\mathfrak{q}}}^{S}({\mathbb{C}}_{\lambda-\rho({\mathfrak{u}})})
$
 in the weakly fair range, 
 but implies
 ${\mathcal{R}}_{{\mathfrak{q}}}^{S-i}({\mathbb{C}}_{\lambda-\rho({\mathfrak{u}})})=0$
 for $i \ne 0$
and the unitarizability of ${\mathcal{R}}_{{\mathfrak{q}}}^{S}({\mathbb{C}}_{\lambda-\rho({\mathfrak{u}})})$
 unless it vanishes.  
In particular,
 the left-hand side of \eqref{eqn:Bla}
 is equal to $\dim\operatorname{Hom}_K(\pi, {\mathcal{R}}_{\mathfrak{q}}^S({\mathbb{C}}_{\lambda-\rho({\mathfrak{u}})}))$.

Let us compute the alternating sum in the right-hand side of \eqref{eqn:Bla}.   We recall $(K, L \cap K)=(SO(2) \times SO(2m), {\mathbb{T}} \times U(m))$.  
The first factor $SO(2) \simeq {\mathbb{T}}$ sits in the center of $K$, 
 hence it does not affect the computation of the Weyl group 
 $W({\mathfrak{k}}_{\mathbb{C}}, {\mathfrak{t}}_{\mathbb{C}})$ below.  
We recall 
 ${\mathfrak{u}} \cap {\mathfrak{k}}_{\mathbb{C}}$ from \eqref{eqn:uk}, 
 and set 
\begin{align}
\notag
\Delta^+(w) :=& \Delta^+({\mathfrak{k}}_{\mathbb{C}}, {\mathfrak{t}}_{\mathbb{C}}) \cap w \cdot \Delta^-({\mathfrak{k}}_{\mathbb{C}}, {\mathfrak{t}}_{\mathbb{C}}),
\\
\notag
\ell(w):=&\# \Delta^+(w), 
\\
\label{eqn:Wlk}
W_K^{{\mathfrak{l}} \cap {\mathfrak{k}}}:=& \{w \in W({\mathfrak{k}}_{\mathbb{C}}, {\mathfrak{t}}_{\mathbb{C}})
:
 \Delta^+(w) \subset \Delta({\mathfrak{u}} \cap {\mathfrak{k}}_{\mathbb{C}})\},
\\
\notag
=& \{w \in W({\mathfrak{k}}_{\mathbb{C}}, {\mathfrak{t}}_{\mathbb{C}})
:
 \text{$w \mu$ is dominant for $\Delta^+({\mathfrak{l}}_{\mathbb{C}} \cap {\mathfrak{k}}_{\mathbb{C}}, {\mathfrak{t}}_{\mathbb{C}})$}
\\
\notag
&
\hphantom{MMMMMMM} \text{whenever $\mu$ is dominant for $\Delta^+({\mathfrak{k}}_{\mathbb{C}}, {\mathfrak{t}}_{\mathbb{C}})$}\}.  
\end{align}

Let $\mu_0 \in {\mathbb{Z}}$ and $\mu \in {\mathbb{Z}}^m$ satisfy
 $\mu_1 \ge \cdots \ge \mu_{m-1} \ge |\mu_m|$.  
For $\pi = {\mathbb{C}}_{\mu_0} \boxtimes F^{SO(2m)}(\mu)$, 
 the $j$-th cohomology group $H^j({\mathfrak{u}} \cap {\mathfrak{k}}_{\mathbb{C}}, \pi)$
 is isomorphic to 
\begin{equation}
\label{eqn:KBBW}
 {\mathbb{C}}_{\mu_0} \boxtimes \bigoplus_{\substack{w \in W_K^{{\mathfrak{l}} \cap {\mathfrak{k}}}
\\ \ell(w) =j
}}
F^{L \cap K}(w(\mu+\rho_c)-\rho_c) 
\end{equation}
as $(L \cap K)$-modules 
by Kostant's Borel--Weil--Bott theorem.

On the other hand, 
since ${\mathfrak{u}} \cap {\mathfrak{p}}_{\mathbb{C}}$ is isomorphic 
 to ${\mathbb{C}}_1 \boxtimes {\mathbb{C}}^m$
 as an $L \cap K \simeq {\mathbb{T}} \times U(m)$-module, 
 one has 
\begin{equation}
\label{eqn:Sup}
S({\mathfrak{u}} \cap {\mathfrak{p}}_{\mathbb{C}}) 
 \simeq
  \bigoplus_{\ell=0}^{\infty} {\mathbb{C}}_{\ell} \boxtimes S^{\ell}({\mathbb{C}}^m)
=
\bigoplus_{\ell=0}^{\infty} {\mathbb{C}}_{\ell} \boxtimes F^{U(m)}(\ell, 0, \dots, 0).  
\end{equation}

To compute the alternating sum 
 in the right-hand side of \eqref{eqn:Bla}, 
 we compare irreducible $(L \cap K)$-modules
 occurring in \eqref{eqn:KBBW}
 and \eqref{eqn:Sup}.  
Then the following combinatorial result plays a key role.

\begin{lemma}
\label{lem:23021125}
For $\mu_0 \in {\mathbb{Z}}$,
 $\mu \in {\mathbb{Z}}^m$ satisfying
 $\mu_1 \ge \cdots \ge \mu_{m-1} \ge |\mu_m|$
 and $w \in W_K^{{\mathfrak{l}} \cap {\mathfrak{k}}}$, 
 the following two conditions are equivalent:
\par\noindent
{\rm{(i)}}\enspace
$\mu_0=\ell+m-1$, 
 $w (\mu+\rho_c)-\rho_c=(\ell, 0, \dots,0)$,  
\par\noindent
{\rm{(ii)}}\enspace
$w =e$, $\mu_0=\ell+m-1$, 
 and $\mu=(\ell, 0, \dots,0)$.  
\end{lemma}

\begin{proof}
Let us verify (i) $\Rightarrow$ (ii).  
By the definition \eqref{eqn:Wlk}, 
 any element $w \in W_K^{{\mathfrak{l}} \cap {\mathfrak{k}}}$
 is of the form
\[w\mu=(\mu_1, \dots, \widehat{\mu_{j_1}}, \dots, \widehat{\mu_{j_{2r}}}, \dots, \mu_m, -\mu_{j_{2r}}, \dots, -\mu_{j_1})
\]
 for some $1 \le j_1 < j_2 < \cdots < j_{2r} \le m$
 with $0 \le r \le \frac m 2$.  
If $r>0$, 
 then the last component of $w(\mu+\rho_c)-\rho_c$ amounts to 
 $-(\mu_{j_1}+m-j_1)$ 
 which is negative, 
 whence $w(\mu+\rho_c)-\rho_c \ne (\ell, 0, \ldots, 0)$.  
Thus the implication (i) $\Rightarrow$ (ii) follows.  

The converse is implication (ii) $\Rightarrow$ (i) is obvious.  
\end{proof}

\begin{proof}
[Proof of Proposition \ref{prop:23021120}]
Suppose $\pi$ is a $K$-type
 of the form 
$
   {\mathbb{C}}_{\mu_0}\boxtimes \Kfdrep{SO(2m)}{\mu_1, \dots, \mu_m}
$
 with $\mu_1 \ge \cdots \ge \mu_{m-1} \ge |\mu_m|$.  
Then it follows from \eqref{eqn:KBBW}
 and Lemma \ref{lem:23021125}
 that the right-hand side of \eqref{eqn:Bla} equals
\[
\begin{cases}
1
\qquad
&\text{if $\mu_0+1-m \in {\mathbb{N}}$ and $\mu=(\mu_0+1-m, 0, \dots, 0)$, }
\\
0
&\text{otherwise}.  
\end{cases}
\]

Thus the proposition is shown by \eqref{eqn:DR}.  
\end{proof}

As we have mentioned, 
 the general theory \cite{xvr}
 of Zuckerman's derived functor does not guarantee
 the irreducibility of ${\mathcal{R}}_{{\mathfrak{q}}}^{S}({\mathbb{C}}_{\lambda-\rho({\mathfrak{u}})})$
 because $\lambda-\rho({\mathfrak{u}})$ is not 
 in the good range, 
 namely, 
 $\langle \lambda{\bf{1}}_{m+1}-\rho({\mathfrak{u}})+\rho_{\mathfrak{l}}, \alpha \rangle$
 is not necessarily positive
 for all $\alpha \in \Delta({\mathfrak{u}})$
 as $\lambda{\bf{1}}_{m+1}-\rho({\mathfrak{u}})+\rho_{\mathfrak{l}}
 =(m-1,m-2,\dots, 0,-1)$.  
Nevertheless, 
 in our specific setting, 
 one has the following irreducibility result:

\begin{proposition}
\label{prop:irr}
The $G$-module $\dcoh{m(m-1)}{G/L, {\mathcal{L}}_{m-1}}$ is
 non-zero and irreducible.  
\end{proposition}

\begin{proof}
Our proof utilizes the multiplicity-free $K$-type formula
 in Proposition \ref{prop:23021120}.  
Let $W_{\ell}:={\mathbb{C}}_{\ell+m-1} \boxtimes \Kfdrep {SO(2m)}{\ell, 0, \dots, 0}$.  
Suppose $V$ is an irreducible $({\mathfrak{g}}_{\mathbb{C}}, K)$-submodule
 in $\dcoh{m(m-1)}{G/L, {\mathcal{L}}_{m-1}}_{K\operatorname{-finite}}$.  
Since the $K$-type formula in Proposition \ref{prop:23021120}
 is multiplicity-free, 
 the $K$-type of $V$ is of the form
 $\oplus_{\ell \in J} W_{\ell}$
 for some subset $J \subset {\mathbb{N}}$.  
We shall show $J={\mathbb{N}}$.  
Assume it were not the case.  
Since the underlying $({\mathfrak{g}}_{\mathbb{C}}, K)$-module
 of the $G$-module $\dcoh{m(m-1)}{G/L, {\mathcal{L}}_{m-1}}$ is unitarizable, 
 it is completely reducible.  
Therefore, 
 by replacing $J$ with ${\mathbb{N}} \setminus J$
 if necessary, 
 we may find $N \in J$
 such that $N+1 \not \in J$.  
Then $V \cap \oplus_{\ell \le N} W_{\ell}$ would be 
 a $({\mathfrak{g}}_{\mathbb{C}}, K)$-submodule 
 of $V$
 because the $SO(2m)$-type
 in 
$
   {\mathfrak{p}}_{\mathbb{C}} W_{\ell}
 :={\mathbb{C}}\text{-span}\{X v : X \in {\mathfrak{p}}_{\mathbb{C}}, \, v \in W_{\ell}\}
$
 must be either $W_{\ell+1}$ or $W_{\ell-1}$ 
but $W_{N+1} \not \subset V$
 by the choice of $N$.

On the other hand, 
 the $G$-module $\dcoh{m(m-1)}{G/L, {\mathcal{L}}_{m-1}}$
 cannot have a non-trivial finite-dimensional submodule
 except for the trivial one-dimensional representation
 because it is unitarizable.  
But the trivial one-dimensional representation cannot be 
 a submodule 
 because the ${\mathfrak{Z}}({\mathfrak{g}}_{\mathbb{C}})$-infinitesimal character 
 of the cohomology is $(m-1, m-2, \dots, 0,-1)$
 in the Harish-Chandra parametrization.  
Hence the proposition is proved.  
\end{proof}

\section{Proof of Theorem \ref{thm:main} and Corollary \ref{cor:230217}}

This section completes the proof of our main results.  
We have seen in Section \ref{sec:Pen}
 that the Penrose transform is a $G$-intertwining operator:
\[
  {\mathcal{R}} \colon \dcoh{m(m-1)}{G/L, {\mathcal{L}}_{m-1}} \to 
C^{\infty}(G/K, {\mathcal{V}}_{m-1}).  
\]

Since $\Delta({\mathfrak{u}} \cap {\mathfrak{p}}_{\mathbb{C}})
\subset \Delta({\mathfrak{p}}_+)$, 
 $X$ has a $K$-equivariant holomorphic fiber bundle structure 
 over $C$, 
 one has from \cite{xhs-ellholo, xholoell2} 
 that ${\mathcal{R}}$ is non-zero
 on $W_0={\mathbb{C}}_{m-1} \boxtimes {\mathbb{C}}$, 
 and $\operatorname{Image} {\mathcal{R}}
 \subset {\mathcal{O}}(G/K, {\mathcal{V}}_{m-1})$.

By the irreducibility of the $G$-module
 $\dcoh{m(m-1)}{G/L, {\mathcal{L}}_{m-1}}$ 
 given by Proposition \ref{prop:irr}, 
one sees that ${\mathcal{R}}$ is injective.  
Since ${\mathcal{O}}(G/K, {\mathcal{V}}_{m-1})$ 
 is $K$-multiplicity-free, 
 one obtains the following proposition from 
 Propositions \ref{prop:3.1} and \ref{prop:23021120}.

\begin{proposition}
The Penrose transform 
$
   {\mathcal{R}}
$
 in \eqref{eqn:defPen}
 induces a $({\mathfrak{g}}_{\mathbb{C}}, K)$-isomorphism
 between the underlying $({\mathfrak{g}}_{\mathbb{C}}, K)$-modules
 of $\dcoh{m(m-1)}{G/L, {\mathcal{L}}_{m-1}}$
 and $\Sol (D, \Delta)$.  
\end{proposition}

Now Theorem \ref{thm:main} follows from the general argument
 on the maximal globalization 
 as in \cite{xhseki02}
 because both the $G$-module
 $\dcoh{m(m-1)}{G/L, {\mathcal{L}}_{m-1}}$
 and $\Sol (D, \Delta)$ are the maximal globalizations
 of their underlying $({\mathfrak{g}}_{\mathbb{C}}, K)$-modules.

Since the $G$-module  $\dcoh{m(m-1)}{G/L, {\mathcal{L}}_{m-1}}$
 is irreducible (Proposition \ref{prop:irr}), 
 and its underlying $({\mathfrak{g}}_{\mathbb{C}}, K)$-module
 is a lowest weight module \cite{xhs-ellholo, xholoell2}
 with the $K$-type formula as in Proposition \ref{prop:23021120}, 
 it is identified with one of the two minimal $({\mathfrak{g}}_{\mathbb{C}},K)$-module, 
 see \cite{xtamori}.  
The same argument applies 
 if we replace ${\mathfrak{q}}={\mathfrak{l}}_{\mathbb{C}} + {\mathfrak{u}}$
 by ${\mathfrak{l}}_{\mathbb{C}}+{\mathfrak{u}}^-$.  
Thus Corollary \ref{cor:230217} is also shown.  

\vskip 1pc
\par\noindent
{\bf{Acknowledgement}}
\newline
The article treats a topic 
 which is closely related to the invited talk
 at the session \lq\lq{Harmonic Analysis and Representation Theory}\rq\rq\ in the 29th Nordic Congress of Mathematicians
 held on July, 2023 in Denmark.  
The author would like to express her sincere gratitude 
 to T.\ Kobayashi, P.-E.\ Paradan, M.\ Pevzner, J.\ Frahm,
 B.\ {\O}rsted, B.\ Speh, and G.\ Zhang 
 for their encouragement, interests and comments.

\leftline{Hideko SEKIGUCHI}
\leftline{Graduate School of Mathematical Sciences, }
\leftline{The University of Tokyo,} 
\leftline{Japan.  }

\end{document}